\documentclass[11pt]{amsart}
\usepackage{amsmath,amsthm,amsfonts,amssymb,verbatim,eucal,bm}
\usepackage[all]{xy}
\usepackage{url}

\numberwithin{equation}{section}

\newtheorem{thm}[equation]{Theorem} 
\newtheorem{prop}[equation]{Proposition}
\newtheorem{lemma}[equation]{Lemma} 
\newtheorem{cor}[equation]{Corollary}
\newtheorem{example}[equation]{Example}
\newtheorem{remark}[equation]{Remark}
\newtheorem{definition}[equation]{Definition}

\DeclareMathOperator{\coh}{H}

\DeclareMathOperator{\sgn}{sgn}
\DeclareMathOperator{\Sym}{Sym}
\DeclareMathOperator{\Span}{Span}
\DeclareMathOperator{\Tor}{Tor}

\newcommand{\ontorightarrow}{\twoheadrightarrow}
\newcommand{\ttp}{A\hspace{-.1ex}\otimes_{\tau}\hspace{-.3ex} B}
\newcommand{\ttpp}{(\ttp)}

\newcommand{\DOT}{\setlength{\unitlength}{1pt}\begin{picture}(2.5,2)
               (1,1)\put(2,2.5){\circle*{2}}\end{picture}}
\newcommand{\bu}{\DOT}

\newcommand{\N}{{\mathbb N}}

\renewcommand{\ker}{\mbox{\rm Ker\,}}

\newcommand{\ot}{\otimes}

\newcommand{\Ug}{\mathcal{U}({\mathfrak{g}})}

\newcommand{\Wedge}{\textstyle\bigwedge}

\newcommand{\tsigma}{\sigma}
\newcommand{\tdelta}{\delta}

\begin{document}
\begin{abstract}
We build 
resolutions for general twisted tensor products 
of algebras.
These bimodule and module resolutions 
unify many constructions in the literature and 
are suitable for computing Hochschild (co)homology and  
more generally Ext and Tor for (bi)modules.
We analyze in detail the case of 
Ore extensions, consequently obtaining 
 Chevalley-Eilenberg resolutions 
for universal enveloping algebras of Lie algebras
(defining the cohomology of
Lie groups and Lie algebras).
Other examples include  semidirect products, crossed products,
Weyl algebras, Sridharan enveloping algebras, and Koszul pairs.
\end{abstract}
\title[Twisted tensor products]
{Resolutions for twisted tensor products}

\date{July 24, 2018.}
\author{A.V.\ Shepler}
\address{Department of Mathematics, University of North Texas,
Denton, Texas 76203, USA}
\email{ashepler@unt.edu}
\author{S.\  Witherspoon}
\address{Department of Mathematics\\Texas A\&M University\\
College Station, Texas 77843, USA}\email{sjw@math.tamu.edu}
\thanks{
The first author was partially supported by
Simons grant 429539.
The second author was partially supported by
NSF grant DMS-1401016.}

\maketitle


\section{Introduction}

Motivated by questions in noncommutative geometry, 
\v{C}ap, Schichl, and Van\v{z}ura~\cite{CSV} introduced a very general 
{\em twisted tensor product} of 
algebras to replace  the (commutative) tensor product.
Their examples included noncommutative 2-tori and crossed products of C$^*$-algebras
with groups. 
Many other algebras of interest
arise as twisted tensor product algebras: crossed products with Hopf algebras,
algebras with triangular decomposition (e.g., universal enveloping algebras
of Lie algebras and quantum groups), 
braided tensor products defined by  R-matrices,
and other biproduct constructions.
In fact, twisted tensor product algebras are rather copious: 
If an algebra is isomorphic to $A\ot B$
as a vector space for two of its subalgebras $A$ and $B$ 
under the canonical inclusion maps, 
then it must be isomorphic to a twisted tensor product
$\ttp$ for some twisting map $\tau: B\ot A\rightarrow A\ot B$
(see~\cite{CSV}).

Modules over a twisted tensor product algebra
arise from tensoring together modules for the individual algebras:
If $M$ and $N$ are modules over algebras $A$ and $B$, respectively,
compatible with a twisting map $\tau$,
then $M\ot N$ adopts the structure of a module over
$\ttp$.
We describe in this note a general 
method to twist together resolutions
of $A$-modules and $B$-modules in order to construct
resolutions for the corresponding modules over
the twisted tensor product $\ttp$.
A similar method works for bimodules. 
In particular,
we twist together resolutions of algebras
over a field to obtain a resolution 
for a twisted tensor product algebra as a bimodule over itself.

We are motivated by a desire to understand deformations 
of twisted tensor products and to
describe the homology theory 
in terms of the homology of the original factor algebras.  
For example, under some finiteness assumptions, 
the Hochschild cohomology of a tensor product of algebras 
is the  tensor product of their Hochschild cohomology rings.
A similar statement is true of the cohomology of augmented algebras. 
Both results hold because the tensor product of projective resolutions
for the factor algebras is a 
projective resolution for the tensor product of the algebras.

In some particular settings, similar homological constructions 
have appeared for modified versions of the tensor product of algebras.
We mention just a few examples.
Gopalakrishnan and Sridharan~\cite{GS} 
constructed  resolutions for modules of Ore extensions. 
Bergh and Oppermann~\cite{BerghOppermann08}  twisted
resolutions when the twisting arises from a  bicharacter on grading groups.
Jara Martinez, L\'opez Pe\~na, and \c{S}tefan~\cite{MLS} worked with Koszul pairs.
Guccione and Guccione~\cite{GG99,GG} built resolutions
for twisted tensor products, in particular crossed products with Hopf algebras,
out of bar and Koszul resolutions of the factor algebras. 
We adapted this last  construction 
in~\cite{quad}
to handle more
general resolutions for the case of skew group algebras
in order to understand deformations.
Walton and the second author generalized these resolutions to 
smash products with Hopf algebras in~\cite{WW1}.

In this paper, we unify many of these previous constructions and
provide methods useful in new settings
for finding resolutions of modules over twisted tensor product
algebras:
We show very generally that projective 
resolutions for bimodules of two factor algebras 
can be twisted together to construct  
a projective resolution for the resulting bimodule 
for the twisted tensor product given some compatibility conditions.
This twisting of resolutions provides an efficient means for computing
and handling Hochschild (co)homology in particular. 
A similar construction applies to projective (left) module resolutions used,
for example, to compute
(co)homology of augmented algebras. 

We verify that many known resolutions may be viewed as twisted 
resolutions in this way, 
including some of those mentioned above.
We give details in the case of  Ore extensions.
In particular, the bimodule 
Koszul resolution of a universal enveloping algebra
$\Ug$ is a twisted resolution
when $\mathfrak g$ is a finite dimensional supersolvable Lie algebra.
More general Lie algebras can be
handled  via triangular decomposition.
Our method also leads to standard resolutions
for Weyl algebras and some Sridharan enveloping algebras.
For an Ore extension, we adapt results of 
Gopalakrishnan and Sridharan~\cite{GS}  to construct 
twisted product resolutions of  modules. 
We thus  regard the Chevalley-Eilenberg complex of $\Ug$ 
as a twisted product resolution.
This defines Lie algebra and Lie group cohomology
in terms of an iterative twisting of resolutions.

In Section~\ref{sec:preliminaries}, we give definitions and some preliminary 
results.
Bimodule twisted tensor product complexes are constructed in 
Section~\ref{sec:construction} and we show they 
give projective resolutions in Theorem~\ref{maintheorem1}.
Section~\ref{sec:Ore} gives applications
to some types of  Ore extensions.
We construct twisted tensor product complexes for
resolving  modules 
in Section~\ref{sec:construction2},
and we show these complexes are projective resolutions 
in Theorem~\ref{maintheorem2}.
Applications to Ore extensions appear in Section~\ref{sec:Ore2}. 

We fix a field $k$ of arbitrary characteristic throughout.
All tensor products are over $k$ unless
otherwise indicated, i.e., $\otimes=\otimes_k$,
and all algebras are $k$-algebras.
Modules are left modules unless otherwise described. 

\section{Twisted tensor product algebras and compatible
resolutions}\label{sec:preliminaries}
In this section, we recall  twisted tensor product algebras
from~\cite{CSV} and define a compatibility condition
necessary for twisting resolutions together.  
Examples include skew group algebras and
crossed products with Hopf algebras~\cite{Mo},   
twisted tensor products given by  bicharacters of grading groups~\cite{BerghOppermann08},
braided products arising from R-matrices~\cite{Manin},   
two-cocycle twists of Hopf algebras~\cite{RadfordSchneider08}, and more.  

Let $A$ and $B$ be associative algebras over $k$
with multiplication maps $m_A:A\ot A\rightarrow A$ and 
$m_B:B\ot B\rightarrow B$ and multiplicative identities $1_A$ and $1_B$,
respectively. 
We write $1$ for the identity map on any set.

\vspace{1ex}

\subsection*{Twisted tensor products}
A {\em twisting map} $$\tau: B \ot A \to A \ot B$$
is a bijective $k$-linear map 
for which 
$
\tau(1_B\ot a) = a\ot 1_B  \ \mbox{ and } \ 
\tau(b \ot 1_A) = 1_A \ot b
$
for all $a\in A$ and $b\in B$,  and
\begin{equation}\label{eqn:taumm}
\tau \circ (m_B \ot m_A) = 
(m_A \ot m_B) \circ (1 \ot \tau \ot 1)
\circ (\tau \otimes \tau) \circ (1 \ot \tau \ot 1)
\end{equation}
as maps $B\ot B\ot A\ot A \rightarrow A\ot B$. 
The 
{\it twisted tensor product algebra}
$\ttp$ is
the vector space $A\otimes B$ 
together with multiplication
$m_\tau$ given by such a twisting map $\tau$.
By~\cite[Proposition/Definition~2.3]{CSV}, the algebra
$\ttp$ is associative. 

Note that the left-right distinction in a twisted
tensor product algebra is artificial since
$A\ot_\tau B \cong B\ot_{\tau^{-1}} A$.
Indeed, one might identify $\ttp$ with the algebra generated
by $A$ and $B$ (so that $A$ and $B$ are subalgebras)
with relations given by Equation~(\ref{eqn:taumm}).

If $A$ and $B$ are $\N$-graded algebras, we take the standard
$\N$-grading on $A\ot B$ and $B\ot A$ and
say a twisting map $\tau$
is {\em strongly graded} if it takes 
$B_j\ot A_i$ to $A_i\ot B_j$ for all $i,j$
following Conner and Goetz~\cite{ConnerGoetz}.
(Note that~\cite{MLS} leave off the adjective {\em strongly}.)
In this case, the twisted tensor product algebra $\ttp$ is $\N$-graded. 

\vspace{2ex}

\begin{example}\label{ex:weyl} 
{\em
The Weyl algebra $\mathcal{W} =k\langle x,y \rangle/(xy-yx-1)$ 
is isomorphic to the twisted tensor product $\ttp$ of $A =k[x]$ and $B =k[y]$
with twisting map $\tau: B\ot A \rightarrow A \ot B$
defined by $\tau(y \otimes x) = x \otimes y - 1\ot 1$.
Likewise, the Weyl algebra $\mathcal{W}_n$ 
on $2n$ indeterminates,
$$
\mathcal{W}_{n}=k\langle x_1,\ldots, x_n, y_1, \ldots, y_n \rangle
/ (x_ix_j-x_jx_i, \ y_i y_j-y_jy_i, \ x_iy_j-y_jx_i-\delta_{i,j}:
1\leq i , j \leq n )\, , 
$$
is isomorphic to a twisted tensor product. 
These are examples of (iterated) Ore extensions, which we consider
in detail in Section~\ref{sec:Ore}. 
}
\end{example}

\vspace{2ex}

\begin{example}\label{ex:sga}
{\em 
A skew group algebra $S\rtimes G$ for a finite group $G$
acting on an algebra $S$ by automorphisms is isomorphic to the
twisted tensor product $kG \ot_{\tau} S$ of the group algebra $kG$
and of $S$. 
The twisting map $\tau$ is defined by
$\tau(s\ot g) =\ g\ot g^{-1}(s)$ for $s\in S$ and 
$g\in G$.  We consider the special case that $S$ is a Koszul algebra 
 at the end of Section~\ref{sec:construction}. 
}\end{example}

\vspace{2ex}

\subsection*{Bimodules over twisted tensor products}
We fix
a twisting map $\tau: B\ot A\rightarrow A\ot B$
for $k$-algebras $A$ and $B$.

\begin{definition}\label{compatible} {\em 
An $A$-bimodule $M$
is  {\em compatible with $\tau$} 
if there is a bijective $k$-linear  map $\tau^{}_{B,M}: B\ot M\rightarrow M\ot B$ 
commuting with the bimodule structure of $M$ and multiplication in $B$,
i.e., as maps on $B\ot B\ot M$ and on $B\ot A\ot M\ot A$,
respectively, 
\begin{align}
\label{compatible2}   
\tau^{}_{B,M}(m^{}_B\ot 1)  &=  (1\ot m^{}_B) (\tau^{}_{B,M}\ot 1)
   (1\ot \tau^{}_{B,M}) \qquad\text{ and } \\
\label{compatible3}
  \tau^{}_{B,M}(1\ot \rho^{}_{A,M}) &=
   (\rho^{}_{A,M}\ot 1)(1\ot 1\ot\tau) 
   (1\ot \tau^{}_{B,M}\ot 1)(\tau\ot 1\ot 1)\, ,
\end{align}
where $\rho^{}_{A,M}: A\ot M\ot A  \rightarrow M$
is the bimodule structure map.
If $A$ is graded and $M$ is a graded $A$-bimodule, we say that
$M$ is {\em compatible with a strongly graded twisting map} $\tau$ if
there is a map $\tau^{}_{B,M}$ as above that takes $B_i\ot M_j$
to $M_j\ot B_i$ for all $i,j$. 
}
\end{definition}

\vspace{2ex}

\begin{remark}{\em 
Note that the above definition applies to $B$-bimodules
as well as $A$-bimodules by reversing the role of $A$ and $B$.
Indeed, we apply the definition to the algebra $B$,
the twisted tensor product $B\ot_{\tau^{-1}} A$, and the twisting
map $\tau^{-1}$ to obtain conditions
for a $B$-bimodule $N$ to be compatible
with $\tau^{-1}$.
We may rewrite these conditions using
the convenient notation $\tau^{}_{N,A}=(\tau^{-1}_{A,N})^{-1}$.
We obtain an equivalent right version of the above definition: 
A given $B$-bimodule $N$ is {\em compatible}
with $\tau^{-1}$ when there is some bijective $k$-linear map
$\tau^{}_{N,A}: N\ot A \rightarrow A\ot N$ satisfying 
\begin{align}
\label{compatible1}   
\tau^{}_{N,A}(1\ot m_A)  &=  (m^{}_A\ot 1) (1\ot \tau^{}_{N,A})
   (\tau_{N,A}\ot 1)\qquad\text{and}\\
\label{compatible4}
\tau^{}_{N,A}(\rho^{}_{B,N}\ot 1) &= 
  (1\ot \rho^{}_{B,N})(\tau\ot 1\ot 1)(1\ot \tau^{}_{N,A}\ot 1)
   (1\ot 1\ot \tau)\, ,
\end{align}
as maps on $N\ot A\ot A$ and on $B\ot N\ot B\ot A$, respectively,
where
$\rho^{}_{B,N}: B\ot N\ot B \rightarrow N$ is the bimodule structure map.
}\end{remark}

\vspace{2ex}

In light of the last remark,
we will say a bimodule is {\em compatible with $\tau$} 
when it is either an $A$-bimodule compatible with $\tau$
or a $B$-bimodule compatible with $\tau^{-1}$,
since one often identifies $A\ot_\tau B$ and the isomorphic algebra
$B \ot_{\tau^{-1}} A$ in practice.

\vspace{2ex}

\begin{remark}{\em
An $A$-bimodule $M$
is  {compatible with the twisting map $\tau$}
exactly when 
there is a bijective $k$-linear  map
$\tau^{}_{B,M}: B\ot M\rightarrow M\ot B$ 
making the following diagram commute:
\vspace{1ex}
\begin{small}
\begin{equation}\label{double-commutative-diagram}
\entrymodifiers={+!!<0pt,\fontdimen22\textfont2>}
\xymatrixcolsep{1ex}
\xymatrixrowsep{7ex}
\xymatrix{
&&
B\ot M\ot B 
\ar[drr]^{\ \tau^{}_{B,M}\ot 1}
&&
\\
B\ot B\ot M
\ar[urr]^{\ 1\ot \tau^{}_{B,M}\ \ }
\ar[dr]_{\ m^{}_B\ot 1\ \ }
&&&&  
M \ot B \ot B
\ar[dl]^{\ \ 1\ot m^{}_B}
\\ 
&B\ot M
\ar[rr]_{\tau^{}_{B,M}} 
&&
M\ot B
&
\\
B\ot A\ot M \ot A
\ar[dr]_{\tau\ot 1\ot 1\ \ \ }
\ar[ur]^{1\ot\rho^{}_{A,M}\ \ \ }
&& &&
A\ot M\ot A\ot B
\ar[ul]_{\ \ \ \rho^{}_{A,M}\ot 1}
\\
&
A\ot B\ot M\ot A
\ar[rr]_{\ 1\ot\, \tau^{}_{B,M}\ot 1\ }
&&
A\ot M\ot B\ot A
\ar[ur]_{\ \ \ 1\ot 1\ot \tau}
&
.
}
\end{equation}
\end{small}
A similar diagram expresses compatibility of a $B$-bimodule $N$ with $\tau$.}
\end{remark}

\vspace{2ex}

\begin{example}{\em 
Let $M=A$, an $A$-bimodule via multiplication.
Then $A$ is compatible with~$\tau$ via   $\tau_{B,A}=\tau$.
Similarly $N=B$ is compatible with $\tau$.
} 
\end{example}

\vspace{1ex}

\subsection*{Bimodule structure}
When $M$ and $N$ are compatible with $\tau$, 
the tensor product $M\ot N$ is naturally an
$\ttp$-bimodule via the following composition of maps:
\begin{equation}\label{eqn:bimod-str}
\begin{split}
\begin{xy}*!C\xybox{
\xymatrix{
A\! \ot_{\tau} \! B\! \ot\!  M\! \ot \! N\!\ot \! A \! \ot_{\tau} \! B 
\ar[rrr]^{1\ot\, \tau_{B,M}\ot\, \tau_{N,A}\ot 1}
&&& A\! \ot \! M\! \ot \! B\! \ot \! A\! \ot \! N\! \ot\! B && \\
 {\hspace{3.5cm}}\ar[rrr]^{1\ot 1\ot\, \tau\, \ot 1\ot 1 }
&&& A \! \ot \! M\! \ot\! A\! \ot \! B\! \ot\!  N\! \ot\! B 
\ar[rr]^{\hspace{1.25cm}\rho^{}_{A,M}\ot\, \rho^{}_{B,N}}
&& M\ot N \, .
}}
\end{xy}
\end{split}
\end{equation}

\vspace{2ex}

\subsection*{Bimodule resolutions}
For any $k$-algebra $A$, let $A^e=A\ot A^{op}$
be its enveloping algebra,
with $A^{op}$ the opposite algebra to $A$.
We view an $A$-bimodule $M$ as a left $A^e$-module. 
In Lemma~\ref{lem:construction} below,  we
will construct 
a projective  resolution of an $\ttpp ^e $-module $M\ot N$
from individual resolutions of $M$ and $N$
using some compatibility conditions.
Let $P_{\DOT}(M)$ 
be an $A^e$-projective resolution of $M$
and let $P_{\DOT}(N)$ be a $B^e$-projective resolution of $N$:
\begin{align}
\label{eqn:Ares}  \cdots\rightarrow P_2(M)\rightarrow P_1(M) \rightarrow P_0(M)
\rightarrow M\rightarrow 0&, \\
 \label{eqn:Bres}  \cdots\rightarrow P_2(N)\, 
\rightarrow P_1(N)\,\rightarrow P_0(N)\,
\rightarrow N\, \rightarrow 0&.
\end{align}

\vspace{2ex}

\subsection*{Bar and reduced bar resolutions}
For example, $M$ could be $A$ itself and $P_{\DOT}(A)$ could be the {\em bar resolution} 
$\text{Bar}_{\DOT}(A)$ given by
$\text{Bar}_n(A) =A^{\ot (n+2)}$ with differential
$$
  a_0\ot a_1\ot\cdots\ot a_{n+1}\ \mapsto\ 
\sum_{i=0}^n (-1)^i a_0\ot\cdots\ot a_i a_{i+1}\ot\cdots\ot a_{n+1}
$$ 
for all $n\geq 0$ and $a_0,a_1,\ldots, a_{n+1}\in A$. 
We also use the {\em reduced bar resolution} 
$\overline{\text{Bar}}_{\DOT}(A)$ with
$\overline{\text{Bar}}_{n}(A)= A\ot \bar{A}^{\ot n} \ot A$
for $\bar{A}=A/k1_A$
and differential 
given by the same formula.

\vspace{2ex}

\subsection*{Compatibility conditions}
We now define what it means for resolutions 
to be compatible with the twisting map $\tau$.
We tensor arbitrary resolutions (\ref{eqn:Bres}) and (\ref{eqn:Ares}) 
 with $A$ and $B$ on the right and left
to obtain complexes 
$$P_{\DOT}(N)\ot A,\ A\ot P_{\DOT}(N),\
P_{\DOT}(M)\ot B,\ \text{ and } B\ot P_{\DOT}(M)\, .$$
Viewing these simply as exact sequences of vector spaces, 
we note that any $k$-linear maps 
$\tau^{}_{N,A}: N\ot A \rightarrow A\ot N$ and 
$\tau^{}_{B,M}:B\ot M \rightarrow M\ot B$ 
can be lifted to $k$-linear chain maps
\begin{equation}\label{eqn:tauAB}
\tau^{}_{P_{\DOT}(N),A}: P_{\DOT}(N)\ot A \rightarrow
A\ot P_{\DOT}(N) \ \mbox{ and } \ 
\tau^{}_{B,P_{\DOT}(M)}: B\ot P_{\DOT}(M)\rightarrow P_{\DOT}(M)\ot B\, .  
\end{equation}
For simplicity in the sequel, we will write
$\tau^{}_{i,A}=\tau^{}_{P_i(N),A}$
and 
$\tau^{}_{B,i}=\tau^{}_{B,P_i(M)}$,  
for each $i$,
when no confusion will arise. 
We will use such maps to glue the two resolutions together
provided they satisfy the following compatibility conditions. 
These conditions  just state that the chain maps  commute with multiplication
and with  bimodule structure maps.
There are many settings in which compatible chain maps do exist, 
as we will see. 

\vspace{3ex}

\begin{definition}{\em 
Let $M$ be an $A$-bimodule that is compatible with $\tau$. 
A projective $A$-bimodule resolution $P_{\DOT}(M)$ is
{\em compatible with the twisting map $\tau$} if
each $P_i(M)$ is compatible with $\tau$ via a map 
$$
\tau^{}_{B,i}: B\ot P_i(M)\longrightarrow P_i(M)\ot B\, $$ 
with
$\tau^{}_{B,\DOT}$ a chain map lifting $\tau^{}_{B,M}$.
If $A$ is graded, $M$ is a graded $A$-bimodule, and
$P_{\bu}(M)$ is a graded projective $A$-bimodule resolution, we say
that $P_{\bu}(M)$ is {\em compatible with
a strongly graded twisting map} $\tau$
if there are maps $\tau^{}_{B,i}$ as above
taking $B_j\ot (P_i(M))_l$
to $(P_i(M))_l\ot B_j$ for all $j,l$. 
}
\end{definition}

\vspace{1ex}

\begin{remark}{\em
The above definition applies to $B$-bimodule
resolutions as well as $A$-bimodule resolutions
by reversing the role of $A$ and $B$
in the definition, again as $A\ot_\tau B= B \ot_{\tau^{-1}} A$.
For a $B$-bimodule $N$ that is compatible with $\tau$,  
the definition implies
that a projective $B$-bimodule resolution $P_{\DOT}(N)$ of $N$ is
{\em compatible with the twisting map $\tau$} when
each $P_i(N)$ is compatible with $\tau$ via a map 
$\tau^{}_{i,A}:P_i(N)\ot A \rightarrow A\ot P_i(N)$, with
$\tau^{}_{\DOT,A}$ a chain map lifting $\tau^{}_{N,A}$.
Thus we say a resolution is compatible with $\tau$
if it is either an $A$-bimodule resolution 
or a $B$-bimodule resolution compatible with $\tau$.
}
\end{remark}

\vspace{1ex}

We give some small examples later: 
Example~\ref{ex:weylagain} (Weyl algebra)
and Example~\ref{transvection-group} (skew group algebra).
First a remark on embedding resolutions
and some general results. 

\vspace{1ex}

\begin{remark}\label{embeddingresolutions}
{\em
Note that compatibility is preserved under
embedding of resolutions so long as the extensions
of the twisting map $\tau$ preserve the embedding.
Specifically, assume
$$\phi_{\DOT}:Q_{\DOT}(A) \hookrightarrow P_{\DOT}(A)$$
is an embedding of
resolutions of the algebra $A$,
and $P_{\DOT}(A)$ is compatible with a twisting
map $\tau:B\ot A \rightarrow A \ot B$
via chain maps
$$\tau^{}_{B, i}:B\ot P_i(A)\rightarrow P_i(A)\ot B.$$
If the maps $\tau^{}_{B,i}$ 
preserve the
embedding in the obvious sense that 
each $\tau^{}_{B, i}$ restricts to a surjective 
map $B\ot \text{Im}(\phi)
\ontorightarrow \text{Im}(\phi)\ot B$,
then $Q_{\DOT}(A)$ is compatible with $\tau$
via these restrictions.
}
\end{remark}
    
\subsection*{Compatibility of bar and Koszul resolutions}
If $A$ and $B$ are both Koszul algebras 
and $\tau$ is a strongly graded twisting map, then 
the algebra $\ttp$ is known to be Koszul (see~\cite[Example~4.7.3]{PP},
\cite[Corollary~4.1.9]{MLS}, or~\cite[Proposition~1.8]{WW2}).
Conner and Goetz~\cite{ConnerGoetz} examine
the situation when $\tau$ is not strongly graded.
We show next that both bar and Koszul resolutions are compatible with 
twisting maps. We always assume our 
Koszul algebras are connected graded
algebras, 
 so that 
 they are quotients of tensor algebras on generating vector spaces in degree~1.
  Note that the roles of
 $A$ and $B$ may be exchanged in the next proposition. 

\vspace{2ex}

\begin{prop}\label{prop:bar-Koszul}
Let $\tau$ be a twisting map for some $k$-algebras 
$A$ and $B$.
\begin{itemize}
\item[(i)] The bar resolution 
$\text{Bar}_{\DOT}(A)$ is compatible with $\tau$.
\item[(ii)]
The reduced bar resolution $\overline{\text{Bar}}_{\DOT}(A)$
is compatible with $\tau$.
\item[(iii)] If $A$ is a Koszul algebra, $B$
  is a graded algebra,
and $\tau$ is strongly graded, then   
the Koszul resolution $\text{Kos}_{\DOT}(A)$ is compatible
with $\tau$.
\end{itemize}
\end{prop} 
\begin{proof}
(i) The bar resolution of $A$ 
may be twisted by repeated application of the map $\tau$, 
e.g., define
$\tau^{}_{B,i}: B\ot A^{\ot (i+2)} \rightarrow A^{\ot (i+2)}\ot B$ 
by 
applying $\tau$ to the first two tensor
factors on the left, then applying $\tau$ to next
two tensor factors, and so on:
$$\tau^{}_{B,i}
=(1\ot\cdots\ot 1\ot\tau) (1\ot\cdots\ot 1\ot \tau\ot 1)
\cdots(1\ot\tau\ot1\ot\cdots\ot 1) (\tau\ot 1\ot\cdots\ot 1).
$$
Then $\text{Bar}_{\DOT}(A)$ is compatible with
$\tau$ via $\tau^{}_{B,i}$, as may be verified directly by
repeated use of equation~(\ref{eqn:taumm}).

(ii) 
Write the terms in the bar complex $\text{Bar}_{\DOT}(A)$ as
$P_{i}=A^{\ot (i+2)}$ for each $i$, and 
define the terms in the reduced bar complex 
$\overline{\text{Bar}}_{\DOT}(A)$
by $\bar{P}_i = A\ot \bar{A}^{\ot i}\ot A$.
For each $i$, let $T_i$ be the kernel of the quotient
map $\text{Bar}_i(A)\rightarrow \overline{\text{Bar}}_i(A)$.
Then $T_{\DOT}$ is a subcomplex of $\text{Bar}_{\DOT}(A)$
and $\overline{\text{Bar}}_{\DOT}(A) \cong \text{Bar}_{\DOT}(A)/T_{\DOT}$.
By definition of twisting map $\tau$,
the multiplicative identity $1_A$ commutes with elements of $B$
under $\tau$, implying that $\tau_{B,i}$ of part (i)
takes $B\ot T_i$ onto $T_i\ot B$ for each $i$.
Let $\bar{\tau}_{B,i}: B\ot \overline{\text{Bar}}_i(A)
\rightarrow \overline{\text{Bar}}_i(A)\ot B$ be the 
corresponding map on quotients.
Then $\overline{\text{Bar}}_{\DOT}(A)$ is compatible with $\tau$
via the maps $\bar{\tau}_{B,i}$.

(iii) 
The proof of~\cite[Proposition~1.8]{WW2} shows  
that the embedding 
$\text{Kos}_{\DOT}(A) \hookrightarrow \text{Bar}_{\DOT}(A)$
of bimodule resolutions
is preserved by the iterated twisting
in part~(i) above
(see Remark~\ref{embeddingresolutions}).
Thus $\text{Kos}_{\DOT}(A)$ satisfies compatibility.
\end{proof}

\vspace{2ex}

We give an example next showing how
Proposition~\ref{prop:bar-Koszul}
can be used for Koszul resolutions
even when the twisting map
$\tau$ is not strongly graded.

\begin{example}\label{ex:weylagain}{\em
As in Example~\ref{ex:weyl}, 
let $\mathcal W$ be the Weyl algebra 
on $x,y$ with $A=k[x]$ and  $B=k[y]$. 
Let $\text{Kos}_{\bu}(A)$ be the Koszul resolution of $A$ as an $A$-bimodule,
$$
   0\rightarrow A\ot V\ot A \stackrel{d_1}{\longrightarrow}
    A\ot A \stackrel{m}{\longrightarrow} A \rightarrow 0 \, ,
    $$
    where $V=\Span_k\{x\} \subset A$,
$d_1(1\ot x\ot 1)=x\ot 1 -1\ot x$, and $m$ is
multiplication.
Let $\bar{\tau}:B \ot V
\rightarrow V \ot B$
be the swap map
$b\ot v \mapsto v \ot b$ for all $b$ in $B$ and $v$ in $V$,
and define 
$$\bar{\tau}^{}_{B,\DOT} :  B\ot \text{Kos}_{\DOT}(A)
\rightarrow \text{Kos}_{\DOT}(A)\ot B$$ 
by iterations of 
$\tau$ and $\bar{\tau}$: 
$$
\begin{aligned}
  &{\bar{\tau}^{}_{B,0}:    B\ot A \ot A}&
  & \overset{\  \tau\ot 1  \ }{\xrightarrow{\hspace*{1.3cm}}} &
  & {A \ot B \ot A} &
  & \overset{\  1 \ot \tau  \ }{\xrightarrow{\hspace*{1.3cm}}} &
  & {A \ot A \ot B} \, ,  \ 
  \quad\text{and}
  \\
  &\bar{\tau}^{}_{B,1}: B\ot A\ot V\ot A &
&\overset{\  \tau\ot 1\ot 1  \ }{\xrightarrow{\hspace*{1.3cm}}}&
& A\ot B\ot V \ot A& 
& 
\\
&&  
&\overset{\ 1\ot\bar{\tau}\ot 1  \ }{\xrightarrow{\hspace*{1.3cm}}}& 
&A\ot V \ot B\ot A& 
&\overset{\ 1\ot1\ot \tau  \ }{\xrightarrow{\hspace*{1.3cm}}}& 
&A\ot V \ot A\ot B \, .
\end{aligned}
$$
Define $\bar{\tau}^{}_{\bu, A} : 
\text{Kos}_{\bu}(B)\ot A\rightarrow A\ot \text{Kos}_{\bu}(B)$ similarly for 
the Koszul resolution $\text{Kos}_{\bu}(B)$ of $B$.
Note that $\tau$ is not strongly graded,
so part (iii) of
Proposition~\ref{prop:bar-Koszul} does not apply even though
both $A$ and $B$ are Koszul algebras.
Instead, we appeal to part~(ii)
and Remark~\ref{embeddingresolutions} after
taking canonical embeddings
$\text{Kos}_{\DOT}(A) \hookrightarrow
\overline{\text{Bar}}_{\DOT}(A)$ 
and $\text{Kos}_{\DOT}(B)\hookrightarrow
\overline{\text{Bar}}_{\DOT}(B)$.
(For example, view  $A\ot V \ot A$ as a subspace of $A\ot \bar{A}\ot A$; the  terms in other degrees are either 0 or the same as in the bar resolution.)
The maps $\bar{\tau}^{}_{B,\DOT}$
and $\bar{\tau}^{}_{\DOT, A}$  above
are the restrictions to $B\ot\text{Kos}_{\DOT}(A)$ and 
$\text{Kos}_{\DOT}(B)\ot A$
of the maps of the same name
in the proof of 
Proposition~\ref{prop:bar-Koszul}(ii)
(after identifying
$V$ with its image under the quotient
map $A \rightarrow \bar{A}$).
In this way, we see that the Koszul resolutions
$\text{Kos}_{\DOT}(A)$ and 
$\text{Kos}_{\DOT}(B)$ are compatible with the twisting map $\tau$
via $\bar{\tau}^{}_{B,\DOT} $
and
$\bar{\tau}^{}_{\DOT,A}$.
We extend these ideas in Theorem~\ref{simpleOre}.
}\end{example}

\vspace{1ex}


\section{Twisted
product resolutions for Bimodules}\label{sec:construction}
Again, we fix $k$-algebras $A$ and $B$ 
with a twisting map $\tau: B\ot A\rightarrow A\ot B$ 
and consider an $A$-bimodule $M$ and $B$-bimodule $N$.
We build a projective  resolution of $M\ot N$ as
a bimodule over $\ttp$ from resolutions
$P_{\DOT}(M)$ and  $P_{\DOT}(N)$  under our compatibility assumptions. 
We give the construction in Lemma~\ref{lem:construction}, prove exactness
in Lemma~\ref{lem:exact}, and 
show
in Lemma~\ref{lem:proj} that the modules in the construction 
are indeed projective under an additional assumption. 

\vspace{1ex}

\begin{lemma}\label{lem:construction}
Let $M$ be an $A$-bimodule and let $N$ be a $B$-bimodule, both compatible
with a twisting map $\tau$. 
Let $P_{\DOT}(M)$ and $P_{\DOT}(N)$
be projective $A$- and $B$-bimodule resolutions of $M$ and $N$, respectively, 
that are compatible with  $\tau$.
For each $i,j\geq 0$, let 
\begin{equation}\label{xij}
 X_{i,j}= P_i(M)\ot P_j(N) \, ,  
\end{equation}
an $\ttp$-bimodule via diagram~(\ref{eqn:bimod-str}). 
Then $X_{\DOT, \DOT}$ is a 
bicomplex of $\ttp$-bimodules with horizontal and vertical differentials
given by 
$  d^h_{i,j}= d_i\ot 1$ and 
$d^v_{i,j}= (-1)^i\ot d_j$,
where $d_i$ and $d_j$ denote the differentials of the appropriate resolutions:
\end{lemma}
\begin{small}
$$
\entrymodifiers={+!!<0pt,\fontdimen22\textfont2>}
\xymatrixcolsep{6ex}
\xymatrixrowsep{4ex}
\xymatrix{
\vdots \ar[d] & &
\vdots \ar[d] & &
\vdots \ar[d] \\
X_{0, 2} \ar[d]^{d^v_{0,2}} & & 
X_{1, 2} \ar[ll]_{d^h_{1, 2}} \ar[d]^{d^v_{1,2}}  & &
X_{2, 2} \ar[ll]_{d^h_{2, 2}}  \ar[d]^{d^v_{2,2}}  & 
{\cdots} 
\ar[l]  \\ 
{X_{0, 1}} \ar[d]^{d^v_{0, 1}} & &
X_{1, 1} \ar[ll]_{d^h_{1,1}}   \ar[d]^{d^v_{1, 1}} & &
X_{2, 1} \ar[ll]_{d^h_{2,1}} \ar[d]^{d^v_{2, 1}} & 
{\cdots} 
\ar[l]
\rule{0ex}{6ex} 
\\ 
X_{0, 0}  & &
X_{1, 0} \ar[ll]_{d^h_{1,0}} & & 
X_{2, 0} \ar[ll]_{d^h_{2,0}} & 
\cdots\ar[l] 
}
$$
\end{small}

\begin{proof}
The $k$-vector spaces $X_{i,j}$ form a tensor product bicomplex with
differentials as stated. 
The bimodule action of $\ttp$ on $X_{i,j}$ 
commutes with the horizontal and vertical differentials
since $\tau^{}_{\DOT,B}$ and $\tau^{}_{A,\DOT}$ are chain maps.
Therefore this is an $\ttp$-bimodule bicomplex. 
\end{proof}

\vspace{2ex}

\begin{definition}\label{def:tpr} {\em
The {\em twisted product complex} $X_{\DOT}$ 
is the total complex of $X_{\DOT, \DOT}$, i.e., when augmented by $M\ot N$, it is the complex 
\begin{equation}\label{resolution-X}
  \cdots\rightarrow X_2\rightarrow X_1\rightarrow X_0\rightarrow M\ot N\rightarrow 0
\end{equation}
 with $X_n = \oplus_{i+j=n} X_{i,j}$,  
and $n$th differential 
$\sum_{i+j=n} d_{i,j}$ where 
$d_{i,j}= d_i\ot 1 + (-1)^i \ot d_j$.
}\end{definition}

\begin{lemma}\label{lem:exact}
The twisted product complex~(\ref{resolution-X})
is exact.
\end{lemma}

\begin{proof}
By the K\"unneth Theorem~\cite[Theorem~3.6.3]{W},
for each $n$ there is  an exact sequence
$$
\begin{aligned}
0
\longrightarrow \bigoplus_{i+j=n} 
\coh_i\big(P_{\DOT}(M)\big)
 \ot 
  \coh_j\big(P_{\DOT}(N)\big)
 &\longrightarrow 
 \coh_n\big(P_{\DOT}(M)
 \ot P_{\DOT}(N)\big)\\
&\longrightarrow \bigoplus_{i+j=n-1} 
\Tor^{k}_1\Big(\coh_i\big(P_{\DOT}(M)\big),
  \coh_j \big(P_{\DOT}(N)\big)\Big)
\longrightarrow 0\, .
\end{aligned}
$$
Now $P_{\DOT}(M)$ and $P_{\DOT}(N)$ are exact other than in
degree 0, where they have homology $M$ and $N$, respectively. 
Therefore  
$$\coh_i\big(P_{\DOT}(M)\big)=0 \text{ for all } i>0 \ \text{ and } \   
  \coh_j\big(P_{\DOT}(N))=0 \text{ for all } j>0\, .
$$
The Tor term is 0 since $k$ is a field.
Thus for all $n>0$, 
$
\coh_n 
\big(P_{\DOT}(M)\ot P_{\DOT}(N) \big)
=0 , 
$
and
$$
\coh_0\big(P_{\DOT}(M) \ot P_{\DOT}(N) \big)
\cong 
\coh_0\big(P_{\DOT}(M)\big)
 \ot 
\coh_0\big(P_{\DOT}(N)\big)
\cong M\ot N \, 
$$
as vector spaces.
Thus the complex~(\ref{resolution-X}) is exact. 
\end{proof}

In practice, one often can show directly that
each $X_{i,j}$ is projective as an $\ttp$-bimodule,
for example, when working with bar resolutions
and/or Koszul resolutions.
For the general case, we need an extra compatibility assumption, which we explain next.
As each $P_i(N)$ is a projective $B$-bimodule, it  embeds into a free $B^e$-module
$(B^e)^{\oplus J}$ for some indexing set $J$. 
In the following definition, we use the map 
$(\tau\ot 1)(1\ot \tau):
B^e \ot A \rightarrow A \otimes B^e $. 

\begin{definition}{\em 
\label{compatible-with-embedding}
A chain map $\tau^{}_{i,A}: P_{i}(N)\ot A \rightarrow
A\ot P_{i}(N) $
is {\em compatible with a chosen
embedding} 
${P_i(N)} \hookrightarrow
(B^e)^{\oplus J}$
(for some indexing set $J$) if the corresponding diagram is commutative: 
\begin{small}
$$
\entrymodifiers={+!!<0pt,\fontdimen22\textfont2>}
\xymatrixcolsep{6ex}
\xymatrixrowsep{7ex}
\xymatrix{
{P_i(N)}\otimes A\lhook\mkern-7mu \ar[rr] 
\ar[d]_{\tau^{}_{i,A}}
& & 
(B^e)^{\oplus J} \otimes A
\ar[d]^{((\tau\ot 1)(1\ot \tau))^{\oplus J}} 
\\ 
A\otimes {P_i(N)}\lhook\mkern-7mu \ar[rr] 
& & 
A\otimes (B^e)^{\oplus J}\, .
}
$$
\end{small}
Similarly, the map 
$\tau^{}_{B,i}$ of~(\ref{eqn:tauAB}) is 
{\em compatible with a chosen embedding} of $P_i(M)$ into a free
$A^e$-module $(A^e)^{\oplus I}$ (for some indexing set $I$) 
if the corresponding diagram is commutative, i.e., 
if $\tau^{}_{B,i}$ is the restriction of the map $((1\ot\tau)(\tau\ot 1))^{\oplus I}$ to 
$B\ot P_i(M)$.
}  
\end{definition}

\vspace{1ex}

\begin{remark}\label{rem:not}
{\em 
In many settings, one sees directly that each $X_{i,j}$ is
projective, in which case one need not consider this extra compatibility
condition, as the next lemma is not needed.
This is the case, for 
example, when twisting by
a bicharacter on grading groups
(see~\cite[Lemma~3.3]{BerghOppermann08}). 
In other settings, $\tau^{}_{i,A}$ and $\tau^{}_{B,i}$ are automatically compatible
with chosen
embeddings into free modules, for example if $A$ and $B$ are Koszul algebras
and the embeddings are standard embeddings into bar resolutions
(see~\cite[Proposition~1.8]{WW2}). 
}\end{remark}

\vspace{2ex}

\begin{example}{\em
As in Examples~\ref{ex:weyl} and~\ref{ex:weylagain}, 
let $\mathcal W \cong \ttp$ be the Weyl algebra 
on $x,y$, $A=k[x]$, and $B=k[y]$. 
By construction, each map $\bar{\tau}^{}_{i,A}$
is compatible with the canonical embedding
$\text{Kos}_i(A)\hookrightarrow \overline{\text{Bar}}_i(A)$
and  likewise
${\bar{\tau}}^{}_{B,i}$ is compatible with 
$\text{Kos}_i(B)\hookrightarrow \overline{\text{Bar}}_i(B)$.
}\end{example}

\vspace{2ex}

\begin{lemma}\label{lem:proj}
If $\tau^{}_{B,i}$ and $\tau^{}_{j,A}$ are compatible with chosen embeddings
of $P_i(M)$ and $P_j(N)$ into free modules, then $X_{i,j}=P_i(M)\ot P_j(N)$ is a 
projective $\ttp$-bimodule.
\end{lemma} 

\begin{proof}
First we verify the lemma in  case  $P_i(M) = A^e$, $P_j(N) = B^e$,
and the chosen embeddings are the identity maps.
In this case, 
$X_{i,j}= A^e\ot B^e =A\ot A^{op}\ot B\ot B^{op}$.
One checks that the map 
$$
  1\ot \tau\ot 1 : \ A\ot B\ot (A\ot B)^{op}
   \longrightarrow A\ot A^{op}\ot B\ot B^{op}
$$
is an isomorphism of $\ttpp^e$-modules by equation~(\ref{eqn:taumm})
and the definition of the action given in the proof of Lemma~\ref{lem:construction}. 
If  $P_i(M)$ and $P_j(N)$ are arbitrary free modules, and the chosen
embeddings are identity maps, we apply the above map to each summand
$A^e\ot B^e$ of $P_i(M)\ot P_j(N)$ to see that $X_{i,j}$ is a free $(\ttp)^e$-module.

Now we consider the general case, including the possibility that at least one of 
$P_i(M)$, $P_j(N)$ is free but the corresponding chosen embedding into a 
(possibly different) free module is not the identity map. 
The first part of the proof together with the compatibility hypothesis
implies that the embedding of $k$-vector spaces $P_i(M)\ot P_j(N) \hookrightarrow
(A^e)^{\oplus I}\ot (B^e)^{\oplus J}$ given by the tensor product of
the two embedding maps is a map of $(\ttp)^e$-modules. 
\end{proof}

We combine the lemmas to obtain the following theorem. 

\begin{thm}
\label{maintheorem1}
Let $A$ and $B$ be $k$-algebras, and let 
$\tau:B\ot A \rightarrow A \ot B$ be a twisting map.
Let $M$ be an $A$-bimodule and $N$ a $B$-bimodule  
with projective $A$- and $B$-bimodule 
resolutions $P_{\DOT}(M)$ and $P_{\DOT}(N)$, respectively. 
Assume that $M$, $N$, $P_{\DOT}(M)$, and $P_{\DOT}(N)$ are compatible with $\tau$
and the corresponding 
maps $\tau^{}_{B,i}$ and $\tau^{}_{j,A}$ are compatible with
chosen embeddings of $P_i(M)$ and $P_i(N)$ into free modules.
Then the twisted product complex 
with
$$X_n = \oplus_{i+j=n}\, X_{i,j}
\quad\quad\text{ for }\quad X_{i,j}=P_i(M)\ot P_j(N)
$$
gives a projective resolution of $M\ot N$ as $\ttp$-bimodule:
\begin{equation*}
  \cdots\rightarrow X_2\rightarrow X_1\rightarrow X_0\rightarrow 
  M\ot N \rightarrow 0\, .
\end{equation*}
\end{thm}

\begin{proof}
The result follows 
from Lemmas~\ref{lem:construction}, \ref{lem:exact}, and~\ref{lem:proj}.
\end{proof}

\vspace{1ex}

\begin{remark}{\em
The theorem generally 
unifies known constructions of resolutions in several different
contexts, for example,
twisted tensor products given by bicharacters of grading
groups~\cite{BerghOppermann08},
crossed products~\cite{GG},
skew group algebras (semidirect products) of Koszul algebras
and finite groups~\cite{quad},
and smash products of Koszul algebras with Hopf algebras~\cite{WW1}.
}\end{remark}

Theorem~\ref{maintheorem1} combined with
Proposition~\ref{prop:bar-Koszul} and
Remark~\ref{rem:not} implies that 
a twisted  product resolution of $\ttp$ as a bimodule always exists,
since bar resolutions may always be twisted
(and likewise Koszul resolutions, when one or both of the algebras is 
Koszul,
see also~\cite{MLS,PP,WW2}):

\begin{cor}
Let $A$ and $B$ be $k$-algebras with twisting map $\tau: B\ot A \rightarrow B\ot A$.
The following are projective resolutions of $\ttp$ as a bimodule over itself.
\begin{itemize}
\item
The twisted  product complex of two bar resolutions.
\item
The twisted  product complex of  
two Koszul resolutions
when $A$ and $B$ are Koszul algebras
and $\tau$ is strongly graded.
\item
The twisted  product complex
of one bar resolution and one Koszul resolution
in case one of $A$ or $B$ is Koszul and the other is
graded, for $\tau$ strongly graded. 
\end{itemize}
Moreover, 
bar resolutions may be
replaced by reduced bar resolutions in the above statements. 
\end{cor}

\vspace{1ex}

\subsection*{Examples: Skew group algebras}
We give some details for a class of examples introduced in
Example~\ref{ex:sga}. 
The resolutions  in~\cite{quad} for 
$S\rtimes G$, where $G$ is a finite group 
acting by graded automorphisms on a Koszul algebra $S$, appear different
from but are equivalent to~(\ref{resolution-X}) when $M = kG$ (the group
algebra) and  $N=S$. Note that $kG\ot S$ is isomorphic to
$S\rtimes G$ as an $(S\rtimes G)$-bimodule via the twisting map $\tau$.
In~\cite{quad}, the modules $X_{i,j}$  are given as 
$$
   (S\rtimes G) \ot C_i'\ot D_j'\ot (S\rtimes G) 
$$
where $P_i(kG) = kG\ot C_i'\ot kG$, $P_j(S)= S\ot D_j'\ot S$ are 
free $(kG)^e$- and $S^e$-modules determined by vector
spaces $C_i'$, $D_j'$,
respectively.
We assume $P_i(kG)$ is $G$-graded and the grading is compatible with the
$kG$-bimodule action. 
We assume  $P_j(S)$ is a $kG$-module in such a way that the
differentials are $kG$-module homomorphisms, and this action is compatible with that
of $S$, so that $P_j(S)$ becomes an $S\rtimes G$-module.
Compatibility with $\tau$ follows from these assumptions.
There is an isomorphism of $S\rtimes G$-bimodules,  
$$
   (kG\ot C_i'\ot kG) \ot (S\ot D_j'\ot S)\stackrel{\sim}{\longrightarrow}
   (S\rtimes G)\ot C_i'\ot D_j'\ot (S\rtimes G)\, ,
$$
similar to that used in the proof of~\cite[Theorem~4.3]{quad}, given by 
$$
   g\ot x\ot g'\ot s\ot y\ot s' \mapsto  g((hg')s)
    \ot x\ot (g'y)\ot g' s'
$$ for all $g,g'\in G$, $s,s'\in S$, $x$ in the $h$-component of $C_i'$,
and $y\in D_j'$.

\begin{example}
\label{transvection-group}
{\em
In particular,~\cite[Example~4.6]{quad} involves a resolution
that is neither a Koszul resolution nor a bar resolution
and yet satisfies compatibility.
In that example, $k$ is  a field of positive
characteristic $p$, 
$S=k[x,y]$, and $G=\langle g\rangle$ is a group of order $p$
acting on $S$ by  $g\cdot x =x$,
$g\cdot y = x + y$. 
The resolution $P_{\DOT}(S)$ is the 
Koszul resolution $\text{Kos}_{\DOT}(S)$ of $S$,
$$
   0\rightarrow S\ot \Wedge^2 V \ot S \rightarrow 
    S\ot\Wedge^1 V\ot S\rightarrow S\ot S \rightarrow S \rightarrow 0 \, ,
$$
where $V=\Span_k\{x,y\}$. 
The resolution
$P_{\DOT}(kG)$ is the  bimodule resolution of $kG$,
\begin{equation}
\label{group-alg-resolution-example}
    \cdots \stackrel{\eta\cdot}{\longrightarrow} kG\ot kG
   \stackrel{\gamma\cdot}{\longrightarrow} kG\ot kG
   \stackrel{\eta\cdot}{\longrightarrow} kG\ot kG
   \stackrel{\gamma\cdot}{\longrightarrow} kG\ot kG
 \stackrel{ m}{\longrightarrow} kG \longrightarrow 0 \, ,
\end{equation}
where $\gamma = g\ot 1 - 1\ot g$, $\eta = g^{p-1}\ot 1 + g^{p-2}\ot g
  + \cdots + 1\ot g^{p-1}$, and $m$
is multiplication.
Compatibility follows from
Proposition~\ref{prop:bar-Koszul}(i)
using Remark~\ref{embeddingresolutions}
after taking the standard embedding
$\text{Kos}_{\DOT}(S)\hookrightarrow
\text{Bar}_{\DOT}(S)$
and embedding~(\ref{group-alg-resolution-example}) into $\text{Bar}_{\DOT}(kG)$ (see, e.g.,~\cite{BA}).
}
\end{example}

\section{Bimodule resolutions of Ore extensions}\label{sec:Ore}

Many algebras of interest are Ore extensions of other algebras. 
We show how to twist bimodule 
resolutions for such extensions in this section.

\vspace{1ex}

\subsection*{Ore extensions as twisted tensor products}
Let $R$ be a $k$-algebra and fix a $k$-algebra automorphism  $\sigma$
of $R$.
Let $\delta:R\rightarrow R$ be a left $\sigma$-derivation of $R$, that is,
\begin{equation}\label{def:deriv}
\delta(rs) = \delta(r) s + \sigma(r) \delta(s)
\quad \text{ for all } r,s \in R\, .
\end{equation}
The {\em Ore extension}
$
R[x; \sigma, \delta]
$
is the algebra with underlying vector space $R[x]$ 
and multiplication determined by that
of $R$ and of $k[x]$ and  the additional Ore relation
$$xr = \sigma(r)x + \delta(r)
\quad\text{ for all }
r\in R \, . $$
An Ore extension 
$R[x; \sigma, \delta]$
is thus isomorphic to a twisted tensor product $\ttp$
where $A=R$, $B=k[x]$,
and the twisting map $\tau:B\ot A \rightarrow A\ot B$
satisfies
$$\tau(x\ot r)=\sigma(r)\ot x + \delta(r) \ot 1 
\quad\text{ for all }
r\in R \, . $$

\vspace{1ex}

\subsection*{Free resolutions for iterated Ore extensions}
We will work with general Ore extensions
in Section~\ref{sec:Ore2}. 
Here for simplicity we restrict to the case
that the automorphism on $R$ is the identity,
$\sigma=1_R$, so the Ore relation
sets commutators $xr-rx$ equal to elements in $R$.  
In this case, the Ore extension is also known as a {\em ring of formal differential operators}.  We consider an 
{\em iterated Ore extension} 
$S = (\cdots ( k[x_1][x_2; 1, \delta_2])\cdots ) [x_t; 1, \delta_t]$,
which we abbreviate as 
$$S= k[x_1,\ldots, x_t; \delta_2,\ldots, \delta_t]\, 
= k\langle x_1,\ldots, x_t\rangle / ( x_j x_i - x_i x_j - \delta_j(x_i): 1\leq i < j \leq t)\
$$
with $S\cong k[x_1, \ldots, x_t]$ as a $k$-vector space.
We assume that $S$ is a filtered algebra 
with $\deg(x_i)=1$ for all~$i$.
Then
each $\delta_j$ is a filtered map, i.e.,
$\delta_j(x_i) \in k\oplus k\text{-span}\{x_1, \ldots, x_{j-1} \}$ for $i<j$.
This setting includes Weyl algebras and universal enveloping algebras of 
supersolvable Lie algebras.

\begin{thm}\label{simpleOre}
Consider an iterated Ore extension
$S=k[x_1,\ldots, x_t; \delta_2,\ldots,\delta_t]$ 
with identity automorphisms $\sigma_i=1$
and filtered derivations $\delta_i$.
There is an iterated
twisted product resolution $K_{\DOT}$ 
that is a free resolution of $S$ as a bimodule over itself:
$$K_n= S\ot \Wedge^n V\ot S$$
for $V=k\text{-span}\{x_1, \ldots, x_t\}$ with
differentials given by
(for $1\leq l_1<\cdots <l_n\leq t$)
$$
\begin{aligned}
d_n(&1\ot x_{l_1}\wedge \cdots\wedge x_{l_n}\ot 1) \\
&= \
   \sum_{1\leq i\leq n} (-1)^{i+1} \big(x_{l_i}\ot x_{l_1}\wedge \cdots\wedge\hat{x}_{l_i}\wedge
    \cdots \wedge x_{l_n}\ot 1 
   - 1\ot x_{l_1}\wedge\cdots\wedge \hat{x}_{l_i}\wedge \cdots \wedge x_{l_n}\ot x_{l_i}\big)\\
  & \quad\ + \sum_{1\leq i<j\leq n} (-1)^{j} \ot x_{l_1}\wedge
   \cdots\wedge x_{l_{i-1}}\wedge
   \bar{\delta}_{l_j}(x_{l_i})
   \wedge x_{l_{i+1}}\wedge\cdots 
   \wedge \hat{x}_{l_j}\wedge \cdots
   \wedge x_{l_n}\ot 1 \, ,
\end{aligned}
$$
where $\bar{\delta}_{l_j}(x_{l_i})$ is the image of
${\delta_{l_j}(x_{l_i})}$
under the projection $k\oplus V\ontorightarrow V$.

\end{thm}

\begin{proof}
We induct on $t$.
For each $i$, the Koszul resolution of $k[x_i]$ is embedded in the
(reduced) bar resolution of $k[x_i]$ as
\begin{equation}\label{eqn:Koszuli}
  0 \rightarrow k[x_i]\ot \Span_k\{x_i\}\ot k[x_i] \stackrel{d_1}{\longrightarrow}
   k[x_i]\ot k[x_i] \stackrel{m}{\longrightarrow} k[x_i] \rightarrow 0 \, ,
\end{equation}
where $d_1(1\ot x_i\ot 1) = x_i\ot 1 - 1\ot x_i$ and $m$ is multiplication. 
For $t=i=1$, the complex~(\ref{eqn:Koszuli})
is a resolution of $S$ satisfying the statement
of the theorem.

Now assume $t\geq 2$ and that the iterated Ore extension $A=k[x_1,\ldots, x_{t-1}; 
\delta_2,\ldots, \delta_{t-1}]$ has a free bimodule resolution $P_{\DOT}(A)$
as in the theorem. 
Let $B=k[x_t]$ and let $P_{\DOT}(B)$ be the Koszul resolution~(\ref{eqn:Koszuli}) for $i=t$.
Then $S=\ttp$
where 
$$
  \tau(x_t\ot a) = a\ot x_t + \delta_t(a)\ot 1
\quad\text{ for all } a\in A\, .
$$


\subsection*{Embedding into the reduced bar resolution}
We embed $P_{\bu}(A)$ into the reduced bar resolution
$\overline{\text{Bar}}_{\DOT}(A)$
and then define twisting maps for $P_{\bu}(A)$ via this embedding:
Let $\phi_n: P_n(A)\rightarrow A^{\ot (n+2)}$ 
be the standard symmetrization map 
defined by
$$
\phi_n(1\ot x_{l_1}\wedge \cdots\wedge x_{l_n}\ot 1) =
  \sum_{\sigma\in\Sym_n} \sgn \sigma \ot x_{l_{\sigma(1)}}\ot \cdots \ot x_{l_{\sigma(n)}} \ot 1
$$
for all $1\leq l_1<\cdots < l_n \leq t-1$.
This is a chain map from $P_{\bu}(A)$ to
$\text{Bar}_{\DOT}(A)$.
Compose with the quotient map 
${\text{Bar}}_{\DOT}(A)\rightarrow
\overline{\text{Bar}}_{\DOT}(A)$
to obtain a chain map
$$
\bar{\phi}_{\DOT}:
P_{\DOT}(A)\longrightarrow\overline{\text{Bar}}_{\DOT}(A)\,
.
$$
Note that the image of $P_{\bu}(A)$  in the bar resolution
$\text{Bar}_{\DOT}(A)$, under $\phi_{\DOT}$, 
intersects the kernel of this quotient
map trivially. 
Thus the induced map $\bar{\phi}_{\bu}$ is injective.

\subsection*{Iterated twisting}
The reduced bar
resolution is compatibile with $\tau$
via the map
$$
\bar{\tau}^{}_{B,\DOT}:
B\ot\overline{\text{Bar}}_{\DOT}(A)
\longrightarrow
\overline{\text{Bar}}_{\DOT}(A)\ot B\, 
$$
from the
proof of Proposition~\ref{prop:bar-Koszul}(ii).
We argue that $\bar{\tau}^{}_{B,\DOT}$
restricts to a surjective map
$$
\tilde{\tau}^{}_{B,\DOT}:
B\ot P_{\DOT}(A)\longrightarrow P_{\DOT}(A)\ot B$$
by verifying that it preserves the image of $\bar{\phi}_{\DOT}$,
i.e., $\bar{\tau}^{}_{B,n}$ takes
$B\ot \text{Im}(\bar{\phi}_{n})$ 
  onto $\text{Im}(\bar{\phi}_{n}) \ot B$
  for all $n$.
We apply $\bar{\tau}^{}_{B,n}$ to 
$$x_t \ot \bar{\phi}_n(a_0\ot y_1\wedge\cdots\wedge y_n\ot a_{n+1})
=
\sum_{\pi\in \text{Sym}_n} \sgn\pi\
(x_t\ot a_0\ot y_{\pi(1)}\ot\cdots\ot y_{\pi(n)} \ot a_{n+1}),
$$
for some $a_0, a_{n+1}$ in $A$, in order to move $x_t$
to the far right, obtaining
$$
\Big(\sum_{\pi\in \text{Sym}_n} (\sgn\pi)\, 
a_0 \ot y_{\pi(1)}\ot\cdots\ot y_{\pi(n)} \ot a_{n+1} \Big) \ot x_t
\in \text{Im}(\bar{\phi}_{n})\ot B
$$
plus additional terms that arise from the
relation 
$\tau(x_t\ot y_{\pi(i)}) = y_{\pi(i)} \ot x_t
+\delta_t(y_{\pi(i)})\ot 1$.
(We use the same notation for elements of $A$ and their 
images under the quotient map
$A\rightarrow \bar{A}$ in cases where
no confusion can arise.) 
Since $\tau(1\ot y_j)=y_j\ot 1$ for all $j$,
these additional terms sum to
$$
\begin{aligned}
& \sum_{\pi\in \text{Sym}_n} 
(\sgn\pi)\
\delta_t(a_0)  \ot  y_{\pi(1)}\ot\cdots\ot y_{\pi(n)} \ot a_{n+1}  \ot 1
\\
&\ \ \ \quad
+ \sum_{\pi\in \text{Sym}_n} 
\sum_{1\leq i\leq n} 
(\sgn\pi)\
a_0 \ot y_{\pi(1)}\ot\cdots\ot \bar{\delta}_t(y_{\pi(i)})\ot y_{\pi(i+1)} \ot\cdots\ot y_{\pi(n)}\ot 
a_{n+1}\ot 1\
\\
& \hspace{10ex}
+
\sum_{\pi\in \text{Sym}_n} 
(\sgn\pi)\
a_0 \ot  y_{\pi(1)}\ot\cdots\ot y_{\pi(n)} \ot \delta_t(a_{n+1}) \ot 1 
\\
=
&\ \ 
\bar{\phi}_n\big(
\delta_t(a_0)\ot y_{1}\wedge\cdots\wedge y_{n}\ot a_{n+1}\big)\ot 1
\ +\ 
\bar{\phi}_n\big(
a_0\ot y_{1}\wedge\cdots\wedge  y_{n}\ot \delta_t(a_{n+1}) \big)\ot 1
\\
& \quad\quad +
\sum_{1\leq i \leq n}
\bar{\phi}_n\big(
a_0 \ot y_{1}\wedge\cdots\wedge \bar{\delta}_t(y_{i})\wedge y_{i+1}\wedge y_{n}\ot a_{n+1}\big) \ot 1 \, 
\in\text{Im}(\bar{\phi}_{n})\ot B\,  .
\end{aligned}
$$
We may replace $x_t$ by $x_t^m$ in the above computation
using induction after noting that
$\tau(x_t^m \ot x_i)
= (1\ot m_B)\tau(x_t \ot \big(\tau (x_t^{m-1} \ot x_i)\big)$ for $i<t$.
The above arguments can be modified to apply to
$\bar{\tau}^{-1}_{B,i}$ as well.
Thus the chain map $\bar{\tau}^{}_{B,\DOT}$ preserves the image of $\bar{\phi}_{\DOT}$
and restricts to
a surjective chain map
$\tilde{\tau}^{}_{B, \DOT}: B\ot P_{\DOT}(A)\longrightarrow P_{\DOT}(A)\ot B$
as claimed.

\subsection*{Compatibility on one side}
The complex
$P_{\DOT}(A)$ inherits compatibility with $\tau$
from the compatibility of the reduced bar complex
$\overline{\text{Bar}}_{\DOT}(A)$ with $\tau$.
Indeed, since 
$\overline{\text{Bar}}_{\DOT}(A)$
is compatible
with $\tau$ via a map $\bar{\tau}^{}_{B,\DOT}$
which 
preserves
the embedding
$\bar{\phi}_{\DOT}: P_{\DOT}(A) \hookrightarrow 
\overline{\text{Bar}}_{\DOT}(A)$, 
the complex
$P_{\DOT}(A)$ is compatible with $\tau$
via
the restriction
$\tilde{\tau}^{}_{B,\DOT}$ of
$\bar{\tau}^{}_{B,\DOT}$ to $B\ot P_{\DOT}(A)$.
(See Proposition~\ref{prop:bar-Koszul}(ii)
and its proof and
Remark~\ref{embeddingresolutions}.)

\subsection*{Compatibility on the other side}
Define a chain map $\tau^{}_{\DOT , A}: P_{\DOT}(B)\ot A\rightarrow A\ot P_{\DOT}(B)$
by setting $\tau^{}_{0,A} = (\tau\ot 1)(1\ot \tau)$
and
$$
  \tau^{}_{1,A}((1\ot x_t\ot 1)\ot x_i) = x_i\ot (1\ot x_t\ot 1)
$$
and then extending (uniquely) to $P_1(B)\ot A$ by requiring that compatibility
conditions~(\ref{compatible1}) and~(\ref{compatible4}) hold.
A calculation shows that $\tau^{}_{\DOT , A}$ is a chain map and that
$P_{\DOT}(B)$ is  compatible with $\tau$. 
By their definitions, $\tau^{}_{0,A}$ and $\tau^{}_{1,A}$ are 
compatible with the  embeddings of $P_0(B)$ and $P_1(B)$
into corresponding terms of the (reduced) bar resolution. 

\subsection*{Twisted product resolution}
By construction, the twisted  product resolution $K_{\DOT}$
arising from $P_{\DOT}(A)$ and $P_{\DOT}(B)$ in degree~$n$ is isomorphic 
to $S\ot \Wedge^n V \ot S$ as an $S$-bimodule via the isomorphisms 
$$
\begin{aligned}
&  A\ot \Wedge^i \Span_k\{x_1,\ldots, x_{t-1}\} \ot A\ot B \ot 
  \Wedge^j \Span_k\{x_t\}\ot B \\
& \hspace{2cm}
  \stackrel{\sim}{\longrightarrow} 
\ A\ot B \ot \Wedge^i \Span_k \{ x_1,\ldots,x_{t-1}\} \ot 
  \Wedge^j \Span_k \{x_t\} \ot A\ot B \, , 
\end{aligned}
$$
for $j=0,1$, 
given by applying $\tau^{-1}$ (properly interpreted for each factor)
to the innermost tensor factors $A$ and $B$. 
We check the differentials:
On $X_{n,0}$, the differential is just that arising from the factor $P_n(A)$.
Now consider on $X_{ n -1,1} $,
again writing $x_{l_i} = y_i$   
for some indices $1\leq l_1<\cdots < l_n \leq t-1$: 
$$
\begin{aligned}
d_n & (1\ot y_{1}\wedge\cdots\wedge y_{{n-1}}\ot 1\ot 1\ot x_t \ot 1) \\
&=    \Big( \sum_{1\leq i\leq n-1} (-1)^{i+1} \big(y_{i}\ot
   y_{1}\wedge \cdots \hat{y}_{i}\wedge \cdots \wedge y_{{n-1}}\ot 1
- 1\ot y_{1}\wedge \cdots\wedge \hat{y}_{i} \wedge \cdots \wedge y_{{n-1}}\ot 
   y_{i} \big) \\
& \hspace{1cm} 
+ \sum_{1\leq i< j \leq n-1} (-1)^{j} \ot 
   y_{1}\wedge\cdots \wedge \bar{\delta}_{j}({y}_{i})
\wedge\cdots\wedge\hat{y}_{j}\wedge
  \cdots \wedge y_{{n-1}}\ot 1 \Big)\ot (1\ot x_t\ot 1)\\
&\hspace{3cm}+(-1)^{n-1} (1\ot y_{1}\wedge \cdots \wedge y_{{n-1}}\ot 1)
  \ot (x_t\ot 1 - 1\ot x_t) \, , 
\end{aligned}
$$
which may be rewritten, under the above isomorphism, as 
$$
\begin{aligned}
\sum_{1\leq i\leq n-1} (-1)^{i+1} & y_{i}\ot  y_{1}
    \wedge\cdots\wedge \hat{y}_{i} \wedge\cdots
   \wedge y_{{n-1}}\ot x_t\ot  1\\
 -\ \sum_{1\leq i\leq n-1}& (-1)^{i+1} \ot  y_{1}\wedge \cdots \wedge \hat{y}_{i}\wedge
   \cdots \wedge y_{{n-1}} \ot x_t \ot y_{i} \\
  & + \sum_{1\leq i< j\leq n-1} (-1)^{j} \ot 
   y_{1}\wedge \cdots\wedge \bar{\delta}_{j}({y}_{i})\wedge\cdots\wedge \hat{y}_{j}\wedge
   \cdots\wedge y_{{n-1}}\ot x_t\ot  1\\
& \hspace{1.3cm}+(-1)^{n-1}x_t\ot y_{1}\wedge\cdots\wedge y_{{n-1}}\ot 1 +
   (-1)^{n}\ot y_{1}\wedge\cdots \wedge y_{{n-1}}\ot x_t\\
& \hspace{2.75cm}  
   + (-1)^n\sum_{1\leq i\leq n-1} 1\ot y_{1}\wedge\cdots\wedge
   \bar{\delta}_t(y_{i})\wedge\cdots\wedge 
   y_{{n-1}}\ot 1 \, . 
\end{aligned}
$$
Once one sets $y_{n}=x_t$, 
identifies $y_{1}\wedge \cdots \wedge y_{{n-1}} \ot x_t$
with $y_{1}\wedge\cdots\wedge y_{{n-1}}\wedge x_t$, and makes other
similar identifications, this
agrees with the differential in the statement.
\end{proof} 

\vspace{1ex}

\subsection*{Examples}
The theorem applies in particular to 
the universal enveloping algebra
$\Ug$ of a finite dimensional solvable Lie algebra  
$\mathfrak g$.  Here, we assume
the underlying field 
$k$ is algebraically closed,
else $\mathfrak g$ should be supersolvable;
see~\cite[1.3.14]{Dixmier} and~\cite[Section 3]{BHZZ}.
The theorem gives a bimodule Koszul resolution of $\Ug$. 
Semisimple Lie algebras can then be handled via triangular decomposition.
Other examples include Weyl algebras and Sridharan enveloping algebras~\cite{Sridharan}.

\section{Twisted product resolutions for (left) modules}\label{sec:construction2}

We now consider a twisted product resolution of left modules
instead of bimodules.  We give the one-sided version
of bimodule constructions in Sections~\ref{sec:preliminaries}
and~\ref{sec:construction}. 
Again, we fix $k$-algebras $A$ and $B$ with a twisting map 
$\tau: B\ot A\rightarrow A\ot B$.
In the constructions below, we consider compatible $A$-modules,
but note that we as easily
could have started with compatible $B$-modules instead of $A$-modules
using the inverse twisting map $\tau^{-1}$ instead of $\tau$
in order to lift (left) modules of $A$ and $B$
to (left) modules of $A\ot B=B\ot_{\tau^{-1}}A$.

Let $M$ be an $A$-module with module structure
map $\rho^{}_{A,M}: A\ot M \rightarrow M$
and recall the multiplication map $m_B: B\ot B\rightarrow B$.

\vspace{1ex}

\begin{definition}{\em 
The $A$-module $M$ is 
{\em compatible with the twisting map $\tau$} if there is a bijective
$k$-linear map $\tau^{}_{B,M}: B\ot M\rightarrow M\ot B$ 
such that
\begin{align}
\label{comp1} 
\tau^{}_{B,M}(m_B\ot 1) & = (1\ot  m_B)(\tau^{}_{B,M}\ot 1)(1\ot \tau^{}_{B,M}) \quad\text{and}\\ 
\label{comp2} 
\tau^{}_{B,M}(1\ot \rho^{}_{A,M}) & =
(\rho^{}_{A,M}\ot 1)(1\ot \tau^{}_{B,M})
   (\tau \ot 1)\, 
\end{align}
as maps on $B\ot B\ot M$ and on $B\ot A\ot M$, respectively. 
}\end{definition}

\vspace{1ex}

Note that this definition is
equivalent to the
commutativity of a diagram similar to~(\ref{double-commutative-diagram}),
where $\rho^{}_{A,M}$ is replaced by a one-sided module
structure map.

Let $N$ be a $B$-module with module structure map $\rho^{}_{B,N}: B\ot N \rightarrow N$. 
In case $M$ is compatible with $\tau$, 
the tensor product $M\ot N$ may be
given the structure of an $\ttp$-module via the following composition of maps:
\begin{equation}\label{eqn:mod-str}
\begin{xy}*!C\xybox{
\xymatrix{
A\! \ot_{\tau}\!  B \ot\!  M\! \ot \! N
\ar[rr]^{1\ot\, \tau^{}_{B,M}\ot 1}
&& A \! \ot \! M\ot \! B\! \ot\!  N 
\ar[rr]^{\hspace{.5cm}\rho^{}_{A,M}\ot \rho^{}_{B,N}}
&& M\ot N \, .
}}
\end{xy}
\end{equation}

Let $P_{\DOT}(M)$ 
be an $A$-projective resolution of $M$
and $P_{\DOT}(N)$ a $B$-projective resolution of~$N$:
\[
\begin{aligned}
 &  \cdots\rightarrow P_2(M) \rightarrow P_1(M)\rightarrow P_0(M)\rightarrow k\rightarrow 0 \, , \\
 &   \cdots \rightarrow P_2(N)\, \rightarrow P_1(N)\, \rightarrow P_0(N)\, \rightarrow k\rightarrow 0 \, .
\end{aligned}
\]

\vspace{1ex}

\begin{definition}{\em 
Let $M$ be an $A$-module that is compatible with $\tau$. 
The projective module resolution $P_{\DOT}(M)$ of the $A$-module $M$
is {\em compatible with the twisting map $\tau$} if each
$P_i(M)$ is compatible with $\tau$ via maps $\tau^{}_{B,i}$ for which
$\tau^{}_{B,\DOT}: B\ot P_{\DOT}(M)\rightarrow P_{\DOT}(M)\ot B$ 
is a $k$-linear chain map lifting
$\tau^{}_{B,M}: B\ot M \rightarrow M\ot B$.
}\end{definition}

Under the assumption of compatibility, we make the following definition. 

\vspace{1ex}

\begin{definition}{\em 
Let $M$ be an $A$-module compatible with $\tau$ and $P_{\bu}(M)$ a projective
resolution of $M$ that is compatible with $\tau$.
Let $N$ be a $B$-module. 
The {\em twisted product complex} $Y_{\DOT}$ is the total complex
of the bicomplex $Y_{\DOT,\DOT}$ defined by 
\begin{equation}\label{eqn:Y}
   Y_{i,j}  = P_i(M)\ot P_j(N) \, , 
\end{equation}
with $A\ot_{\tau} B$-module structure given by the maps $\tau^{}_{B,\bu}$
as in equation~(\ref{eqn:mod-str}) and with
vertical and horizontal differentials given by $d^h_{i,j}= d_i\ot 1$
and $d^v_{i,j}= (-1)^i\ot d_j$. 
In other words, 
 $Y_n = \oplus _{i+j=n} Y_{i,j}$ with $d_n = \sum_{i+j=n} d_{i,j}$ where $d_{i,j}=
d^h_{i,j}+ d^v_{i,j}$. 
}\end{definition} 

\vspace{1ex}

\begin{lemma}\label{augmented-algebra-action}
Assume $M$ and $P_{\DOT}(M)$ are  compatible with $\tau$.
Then the twisted product complex $Y_{\DOT}$ is a complex of $\ttp$-modules.
\end{lemma}
\begin{proof}
Each space $Y_{i,j}$ is given the structure of an $\ttp$-module 
via diagram~(\ref{eqn:mod-str}). 
The differentials are module homomorphisms since $\tau^{}_{B,\DOT}$ is
a chain map. 
\end{proof}

\begin{lemma}\label{lem:exact2}
The twisted product complex $\cdots \rightarrow
 Y_2 \rightarrow Y_1\rightarrow Y_0\rightarrow M\ot N\rightarrow 0$
is exact.
\end{lemma}

\begin{proof}
As in the proof of Lemma~\ref{lem:exact}, 
apply the 
K\"unneth Theorem to obtain $\coh_n(Y_{\DOT}) =0$ for all $n>0$ and
$\coh_0(Y_{\DOT})\cong M\ot N$. 
\end{proof}

We wish to prove in general that the modules $Y_{i,j}$ are projective,
so we make an additional assumption in the next lemma.
Since $P_{\DOT}(M)$ is a projective resolution of $M$
as an $A$-module, each $P_i(M)$ embeds in a free $A$-module $A^{\oplus I}$.
\begin{definition}{\em 
\label{compatible-with-embedding2}
For each $i\geq 0$, the map $\tau^{}_{B,i}$ is {\em compatible with a chosen
embedding} 
$P_i(M) \hookrightarrow A^{\oplus I}$
(for some indexing set $I$) if the corresponding diagram is commutative: 
\begin{small}
$$
\entrymodifiers={+!!<0pt,\fontdimen22\textfont2>}
\xymatrixcolsep{6ex}
\xymatrixrowsep{6ex}
\xymatrix{
B\otimes {P_i(M)}\lhook\mkern-7mu \ar[rr] 
\ar[d]^{\tau^{}_{B,i}}
& & 
B\otimes A^{\oplus I}
\ar[d]^{\tau^{\oplus I}} 
\\ 
{P_i(M)}\otimes B\lhook\mkern-7mu \ar[rr] 
& & 
A^{\oplus I}\otimes B\, .
}
$$
\end{small}
}
\end{definition}

In many settings, one  proves directly that 
the modules $Y_{i,j}$ are projective---e.g.\  the 
Ore extensions in the next section---and so one
does not need this additional compatibility assumption, nor the next lemma.

\begin{lemma}\label{lem:proj2}
For  $i\geq 0$, if 
 $\tau^{}_{B,i}$ is compatible with a chosen embedding of $P_i(M)$ into a free 
$A$-module, then $Y_{i,j}=P_i(M)\ot P_j(N)$ is a projective $\ttp$-module.
\end{lemma} 

\begin{proof}
By the hypothesis, it suffices to prove the lemma in case $P_i(A)=A$
and $P_j(B)=B$.
In that case, $A\ot B$ is the right regular module $\ttp$ by definition, and so is free. 
\end{proof}

\vspace{1ex}

Combining Lemmas~\ref{augmented-algebra-action}, \ref{lem:exact2}, 
and~\ref{lem:proj2}, we obtain
the following theorem.

\begin{thm}\label{maintheorem2}
Let $A$ and $B$ be  $k$-algebras 
with twisting map $\tau:B\ot A \rightarrow A\ot B$. 
Let $P_{\DOT}(M)$ and $P_{\DOT}(N)$ be projective $A$- and $B$-module 
resolutions of $M$ and $N$, respectively.
Assume $M$ and $P_{\DOT}(M)$ are  compatible with $\tau$ 
and that the corresponding maps $\tau^{}_{B,i}$
are compatible with chosen embeddings of $P_i(M)$ into free $A$-modules. Then 
the twisted product complex with 
$$
Y_n = \oplus _{i+j=n} Y_{i,j}
\quad\quad\text{ for }\quad
   Y_{i,j}  = P_i(M)\ot P_j(N) \,  
$$
gives a projective resolution of $M\ot N$ 
as a module over the twisted tensor product $\ttp$:
$$
  \cdots\rightarrow Y_2\rightarrow Y_1\rightarrow Y_0\rightarrow 
  M\ot N \rightarrow 0\, .
$$
\end{thm}

\vspace{1ex}

\subsection*{Examples}
Resolutions that may be constructed in this way include the
Koszul resolution of $k$ for a twisted tensor product of two
Koszul algebras (see the proof of~\cite[Proposition~1.8]{WW2}) and
a resolution for a twisted tensor product of algebras whose
twisting map is given by a bicharacter on grading groups (see~\cite{BerghOppermann08}). 
We give another class of examples in the next section.


\section{Resolutions for Ore extensions}\label{sec:Ore2}
In Section~\ref{sec:Ore}, we considered resolutions
of an Ore extension algebra as a bimodule over itself.  Here, we consider (left) modules
over an Ore extension and show how to construct projective resolutions
of these modules by regarding the Ore extension as a twisted tensor product.
Gopalakrishnan and Sridharan~\cite{GS} studied Ore extensions $R[x; \sigma,\delta]$
in case $\sigma$ is the identity automorphism.  They showed 
that if $M$ is a (left) module over $R[x; 1, \delta]$, then
an $R$-projective resolution of $M$ lifts to an
$R[x; 1, \delta]$-projective resolution.
Here we allow arbitrary automorphisms $\sigma$
of $R$ and give conditions under which
an $R$-projective resolution of an $R[x; \sigma, \delta]$-module $M$
lifts to an
$R[x; \sigma, \delta]$-projective resolution.

Again, let $R$ be a $k$-algebra and $\sigma$ a $k$-algebra automorphism
of $R$.
Let $\delta$ be a left $\sigma$-derivation of $R$ (see~(\ref{def:deriv}))
and consider the Ore extension $R[x; \sigma, \delta]$.
Let $A=R$, $B=k[x]$, and $\tau: B\ot A\rightarrow A\ot B$ be the
twisting map determined by $\tau(x\ot r) = \sigma(r)\ot x + \delta(r)\ot 1$
for all $r\in R$, as in Section~\ref{sec:Ore}, so
that $R[x; \sigma, \delta]$ is the twisted tensor product
$\ttp$.

\vspace{2ex}

\subsection*{Modules over Ore extensions}
Consider an $R[x;\sigma,\delta]$-module $M$.
Assume that on restriction to $R$, there is an isomorphism of $R$-modules,
$\phi: M\stackrel{\sim}{\longrightarrow} M^{\sigma}$, where 
$M^{\sigma}$ is the vector space $M$
with $R$-module action given by $r\cdot_{\sigma} m = \sigma(r)\cdot m$
for all $r\in R$ and $m\in M$. 
Then $M$ is compatible with $\tau$: 
We define $\tau^{}_{B,M}: =B\ot M \rightarrow M\ot B$
by setting
\begin{eqnarray*}
  \tau^{}_{B,M}(1\ot m) & = & m\ot 1 \, , \\
  \tau^{}_{B,M}(x\ot m) &= & \phi(m)\ot x + x m\ot 1
\qquad\text{for all}\quad m\in M
\end{eqnarray*}
and extending  by applying compatibility condition~(\ref{comp1}).
That is, since the algebra $B=k[x]$ is free on the generator $x$, for each  element
$m$ of $M$, 
we may define $\tau^{}_{B,M}(x^n\ot m)$ by applying (\ref{comp1}) to $x \ot x^{n-1}\ot m$. 
We check that (\ref{comp2}) holds for elements of the form
$x\ot r\ot m$, where $r\in R$ and $m\in M$. 
Then a careful induction on the power of $x$ shows that
(\ref{comp2}) holds for all elements of the form $x^n\ot r\ot m$.

For example, if $R[x;\sigma,\delta]$ 
is an augmented algebra with
augmentation $\varepsilon : R[x;\sigma,\delta]\rightarrow k$ for which
$\varepsilon \sigma = \varepsilon$, then $\varepsilon \delta=0$ and 
the field $k$ as a module over $R[x;\sigma,\delta]$  
via $\varepsilon$ has the property that $k\cong k^{\sigma}$,
and so $k$ is compatible with $\tau$.

\vspace{3ex}

\subsection*{Projective resolutions}
Let $P_{\DOT}(M)$ be a projective resolution of $M$ as an $R$-module:
\[
  \cdots \stackrel{d_2}{\longrightarrow} P_1(M)
  \stackrel{d_1}{\longrightarrow} P_0(M) \stackrel{\mu}{\longrightarrow}M
  \rightarrow 0 .
\]
For each $i$, set $P_i^{\sigma}(M)=(P_i(M))^{\sigma}$.
Then 
\[
  \cdots\stackrel{d_2}{\longrightarrow} 
   P_1^{\sigma}(M)\stackrel{d_2}{\longrightarrow} P_{0}^{\sigma}(M)
   \stackrel{ \phi^{-1}\mu}{\relbar\joinrel\relbar\joinrel\longrightarrow} M
 \rightarrow 0
\]
is also a projective resolution of $M$ as an $R$-module. 
By the Comparison Theorem, there is an $R$-module chain map from
$P_{\DOT}(M)$ to $P_{\DOT}^{\sigma}(M)$ lifting the identity map $M\rightarrow M$, 
which we view as 
a $k$-linear chain map 
\begin{equation}\label{eqn:sigma-tilde}
\tsigma_{\DOT}
: P_{\DOT}(M)\rightarrow P_{\DOT}(M)
\end{equation}
with 
${\tsigma}_i (rz) = \tsigma(r){\tsigma}_i(z)$ for all $i\geq 0$, 
$r\in R$, and $z\in P_i(M)$.
We will assume for Theorem~\ref{thm:main-Ore} below that 
each ${\tsigma}_i$ is bijective.
Let $P_{\DOT}(B)$ be the Koszul resolution of $k$ for $B=k[x]$,
\begin{equation}\label{eqn:K}
  0 \rightarrow k[x]\stackrel{x\cdot}{\longrightarrow}
    k[x] \stackrel{\epsilon}{\longrightarrow} k \rightarrow 0 \, ,
\end{equation}
where $\epsilon(x)=0$.
The following two lemmas are proven as in~\cite{GS} (where
Gopalakrishnan and Sridharan proved the special case $\sigma =1$). 
We include details for completeness. 

\vspace{3ex}

\begin{lemma}\label{lem:B}
Let $P$ be a projective $R$-module.
There is an $R[x;\sigma,\delta]$-module structure on $P$ that extends
the action of $R$.
\end{lemma}

\begin{proof}
First consider the case that $P=R$, the left regular module.
Let $x$ act on $R$ by $x\cdot r = \delta(r)$ for all $r\in R$. 
One checks that the action of $xr$ in $R[x; \sigma,\delta]$
agrees with that of $\sigma(r)x + \delta(r)$ on $P$, for all $r\in R$.
Next, if $P$ is a free module, it is a direct sum of copies of $R$,
and $x$ acts on each copy in this way. 
Finally, in general, $P$ is a direct summand of a free $R$-module $F$.
Let $\iota : P\rightarrow F$ and $\pi : F\rightarrow P$ be $R$-module
homomorphisms for which $\pi \iota$ is the identity map.
Define $x\cdot p = \pi (x\cdot \iota(p))$ for all $p\in P$, where the
action of $x$ on $\iota(p)$ is as given previously for a free module.
Again one checks that the actions of $xr$ and of $\sigma(r)x + \delta(r)$
agree, and so $P$ is an $R[x;\sigma,\delta]$-module as claimed.
\end{proof}

\vspace{3ex}

\subsection*{Compatibility requirements}
We will use the next lemma to show that
the resolution $P_{\DOT}(M)$ of $M$ as an $R$-module is compatible with
the twisting map $\tau$ (see Lemma~\ref{lem:last}).
Let $f:M\rightarrow M$ be the function given by the action of $x$ on 
the $R[x;\sigma,\delta]$-module $M$.

\vspace{1ex}

\begin{lemma}\label{lem:C}
There is a $k$-linear chain map ${\tdelta}_{\DOT} : P_{\DOT}(M)
\rightarrow P_{\DOT}(M)$ lifting $f:M\rightarrow M$ 
such that for each $i\geq 0$, 
${\tdelta}_i(rz)=\sigma(r)
{\tdelta}_i(z) + \delta(r)z$
for all $r\in R$ and $z\in P_i(M)$. 
\end{lemma}

\begin{proof}
If $i=0$, let ${\tdelta}'_0$ be the action of $x$ on $P_0(M)$
given by Lemma~\ref{lem:B}.
Then $${\tdelta}_0'(rz) - \sigma(r){\tdelta}_0' (z) = \delta(r)z$$
for $r\in R$, $z\in P_0(M)$. 
One checks that $\mu{\tdelta}_0' - f \mu :
P_0(M)\rightarrow M^{\sigma}$ is an $R$-module homomorphism. 
As $P_0(M)$ is a projective $R$-module, there is an $R$-module homomorphism
${\tdelta}_0'':P_0(M)\rightarrow P_0^{\sigma}(M)$ such that
$\mu{\tdelta}'_0 - f \mu = \mu {\tdelta}''_0$.
Let ${\tdelta}_0 = {\tdelta}_0' -{\tdelta}_0''$.
One may check this satisfies the equation in the lemma.

Now fix $i>0$ and assume there are $k$-linear maps 
${\tdelta}_j: P_j(M)\rightarrow P_j(M)$ such that
${\tdelta}_j (rz) = \sigma(r){\tdelta}_j(z) + \delta(r) z$
and  $d_j{\tdelta}_j
= {\tdelta}_{j-1} d_j$ for all $j$, $0\leq j < i$, and 
 $r\in R$, $z\in P_j(M)$. 
Let ${\tdelta}_i': P_i(M)\rightarrow P_i(M)$ be the action of $x$
on $P_i(M)$ given in Lemma~\ref{lem:B}, so that
${\tdelta}_i'(rz) = \sigma(r) {\tdelta}_i'(z) + \delta(r)z$
for all $r\in R$, $z\in P_i(M)$.
Consider the map
$$
   d_i{\tdelta}_i' - {\tdelta}_{i-1} d_i:
   P_i(M) \longrightarrow P_{i-1}^{\sigma}(M) \, . 
$$
A calculation shows that it is an $R$-module homomorphism.
Since ${\tdelta}_{i-1}$ is a chain map, 
$$
   d_{i-1}( d_i {\tdelta}_i' - {\tdelta}_{i-1} d_i) = 0 \, , 
$$ 
and so the image of $d_i{\tdelta}_i '  - {\tdelta}_{i-1} d_i$
lies in $\ker(d_{i-1})=\text{Im}(d_i)$.
Since $P_i(M)$ is projective as an $R$-module, there
is an $R$-homomorphism ${\tdelta}_i'' : P_i(M)\rightarrow P_i^{\sigma}(M)$
such that $d_i{\tdelta}_i ' - {\tdelta}_{i-1} d_i =
d_i {\tdelta}_i ''$. 
Let ${\tdelta}_i = {\tdelta}_i' - {\tdelta}_i''$,
so that $d_i{\tdelta}_i = {\tdelta}_{i-1} d_i$ by construction. 
One checks that for all $r\in R$ and $z\in P_i(M)$, 
$$  {\tdelta}_i(rz) =  {\tdelta}_i' (rz) - {\tdelta}_i''(rz)
    = \sigma(r) {\tdelta}_i'(z) + \delta(r) z - \sigma(r)
    {\tdelta}_i''(z)
   = \sigma(r) {\tdelta}_i (z)
   + \delta(r) z \, .
$$
\end{proof}

\begin{lemma}\label{lem:last}
The resolution $P_{\DOT}(M)$ is compatible with the twisting map $\tau$.
\end{lemma} 

\begin{proof}
Define $\tau^{}_{B,i}: B\ot P_i(M)\rightarrow P_i(M)\ot B$ 
by
\begin{eqnarray*}
  \tau^{}_{B,i}(1\ot z) & = & z\ot 1 \, , \\
   \tau^{}_{B,i}(x\ot z)& = &{\tsigma}_i(z)\ot x +
    {\tdelta}_i(z)\ot 1
    \qquad\text{for all}\quad z\in P_i(M)\, , 
\end{eqnarray*}
where ${\tsigma}_{\DOT}$ is the chain map of (\ref{eqn:sigma-tilde}), 
${\tdelta}_{\DOT}$ is the chain map of Lemma~\ref{lem:C},
and we extend $\tau^{}_{B,i}$ to $B\ot P_i(M)$ as before 
by requiring that compatibility conditions~(\ref{comp1}) and~(\ref{comp2}) hold.
We check condition~(\ref{comp2}) in one case as an example: 
$$
   \tau^{}_{B,i}(x\ot rz) =  {\tsigma}_i (rz)\ot x + {\tdelta}_i(rz)\ot 1 
    = \sigma(r) {\tsigma}_i(z) \ot x + \sigma(r) {\tdelta}_i(z)\ot 1
    + \delta(r) z\ot 1 \, ,
$$
for all $r\in R$, and $z\in P_i(M)$,
 while on the other hand,
$$
\begin{aligned}
(\rho^{}_{A,i}\ot 1) &(1\ot \tau^{}_{B,i})(\tau \ot 1) (x\ot r\ot z) \\
&=    (\rho^{}_{A,i}\ot 1) (1\ot \tau^{}_{B,i}) (\sigma(r)\ot x\ot z + \delta(r)
    \ot 1 \ot z) \\
  &= (\rho^{}_{A,i}\ot 1) ( \sigma(r)\ot {\tsigma}_i (z) \ot x 
    + \sigma(r)\ot {\tdelta}_i(z)\ot 1 + \delta(r)\ot z\ot 1) \\
  &= \sigma(r) {\tsigma}_i(z)\ot x + \sigma(r) {\tdelta}_i(z)\ot 1
    + \delta(r) z \ot 1 \, ;
\end{aligned}
$$
Condition~(\ref{comp2}) holds for all $x^n\ot rz$ by induction on $n$. 
\end{proof}

\vspace{0ex}

\subsection*{Twisting resolutions for an Ore extension}
We now construct a projective resolution of
$M$ as an $R[x;\sigma,\delta]$-module from a projective
resolution of $M$ as an $R$-module.
We take the twisted product of two resolutions:
the $R$-projective resolution of $M$
and the Koszul resolution~(\ref{eqn:K}) of $k$ as a module
over $B=k[x]$.

\vspace{1ex}

\begin{thm}\label{thm:main-Ore}
Let $R[x;\sigma,\delta]$ be an Ore extension.
Let $M$ be an $R[x;\sigma,\delta]$-module for which $M^{\sigma}\cong M$
as $R$-modules. 
Consider a projective resolution $P_{\DOT}(M)$ of $M$ as an $R$-module
and suppose
that each  map ${\tsigma}_{i}: P_{i}(M)\rightarrow P_{i}(M)$
of~(\ref{eqn:sigma-tilde}) is bijective. 
For each $i\geq 0$, set 
$$
  Y_{i,0} =  Y_{i,1}=P_i(M)\ot k[x] \, 
\quad \mbox{ and }\quad Y_{i,j}=0 \ \ \text{ for all } j>1
$$
as in Lemma~\ref{augmented-algebra-action}. 
Then $Y_{\DOT}$ is a projective 
resolution of $M$ as an $R[x;\sigma,\delta]$-module.
\end{thm}

\begin{proof}
By Lemma~\ref{lem:last}, $P_{\bu}(M)$ is compatible with $\tau$, and so 
by Lemmas~\ref{augmented-algebra-action} and~\ref{lem:exact2}, the complex $\cdots\rightarrow Y_1\rightarrow Y_0\rightarrow M \rightarrow 0$
is an exact complex of $R[x;\sigma,\delta]$-modules. 
We verify directly that each $Y_{i,j}$ is a projective module:
For each $i\geq 0$ and $j=0,1$, 
\begin{equation}\label{eqn:GS-iso}
Y_{i,j} \cong R[x;\sigma,\delta]\ot_R P_i(M) 
\end{equation} 
via the $R[x;\sigma,\delta]$-homomorphism 
given by
$$
R[x;\sigma,\delta]\ot _R P_i(M)  \longrightarrow  Y_{i,j}\, ,
\quad\quad
    x\ot z  \mapsto  {\tsigma}_i(z)\ot x + {\tdelta}_i(z)\ot 1 \, ,
$$
    with inverse map given by
    $$
    z\ot x \mapsto 
    x\ot {\tsigma}_i^{-1}(z) - 1\ot {\tdelta}_i({\tsigma}_i^{-1}(z))\, .
$$
Then $R[x;\sigma,\delta]\ot_R P_i(M)$
is projective since it is a tensor-induced module and 
$R[x;\sigma,\delta]$ is flat over $R$.
\end{proof} 


\vspace{1ex}

\begin{remark}{\em
When $\sigma$ is the identity, the complex $Y_{\DOT}$ is
precisely that of Gopalakrishnan and Sridharan~\cite[Theorem~1]{GS},
under the isomorphism~(\ref{eqn:GS-iso}) above. 
As a specific class of examples, we obtain in this way, via iterated Ore extension,
the Chevalley-Eilenberg resolution 
of the  $\Ug$-module $k$ for a finite dimensional 
supersolvable Lie algebra $\mathfrak g$.
}
\end{remark}

\vspace{3ex}

\noindent
{\bf Acknowledgments.}
The authors thank Andrew Conner and Peter Goetz for
helpful comments on an earlier version of this paper,
and a referee for suggesting improvements.



\end{document}